\documentclass[12pt, twoside,a4paper]{amsart}
\usepackage{amsmath,amscd,amsthm,amssymb,amsxtra,latexsym,epsfig,epic,graphics,mathdots,amssymb,mathrsfs,enumitem}
\usepackage{float}
\usepackage[T1]{fontenc}
\usepackage{makecell}
\usepackage[all]{xy}
\usepackage[utf8]{inputenc}
\usepackage{hyperref}
\hypersetup{
	colorlinks=true,
	linkcolor=blue,
	filecolor=magenta, 
	urlcolor=blue,
}
\usepackage{booktabs,longtable}

\usepackage{dsfont}
\usepackage{comment}
\usepackage[margin=1in]{geometry}

\usepackage{pgf,tikz,pgfplots} 
\pgfplotsset{compat=1.18}
\usetikzlibrary{arrows,tikzmark,quotes,cd}
\tikzset{
  commutative diagrams/.cd, 
  arrow style=tikz, 
  diagrams={>=stealth}
}\tikzcdset{
	labels={font=\everymath\expandafter{\the\everymath\displaystyle}},
}
\tikzcdset{
  cells={font=\everymath\expandafter{\the\everymath\displaystyle}},
}

\usepackage{dynkin-diagrams}

\usepackage{graphicx,color}

\newtheorem{Theorem}{Theorem}[section]

\newtheorem{Question}[Theorem]{Question}

\theoremstyle{definition}

\newcommand{\extp}{{\textstyle \bigwedge}} 

\newcommand{\CC}{{\mathbb C}}
\newcommand{\ZZ}{{\mathbb Z}}

\newcommand{\PP}{{\mathbb P}}
\newcommand{\RR}{{\mathbb R}}

\newcommand{\cO}{{\mathcal O}}

\newcommand{\into}{\hookrightarrow}
\newcommand{\onto}{\twoheadrightarrow}

\DeclareMathOperator{\codim}{{codim}}
\DeclareMathOperator{\Hom}{Hom}

\DeclareMathOperator{\sing}{Sing}

\DeclareMathOperator{\Pic}{Pic}

\DeclareMathOperator{\SO}{SO}

\DeclareMathOperator{\rk}{rk}

\newcommand{\inhom}{{\mathcal H}{\it om}}

\newcommand{\fsl}{\mathfrak{sl}}
\newcommand{\U}{\mathcal{U}}
\newcommand{\Q}{\mathcal{Q}}

\newcommand{\ox}{\mathcal{O}_X}

\newcommand{\OO}{{\mathcal{O}}}

\newcommand{\p}[1]{{\mathbb{P}^{#1}}}

\newcommand{\opn}{{\mathcal O}_{\mathbb{P}^{n}}}
\newcommand{\pn}{{\mathbb{P}^{n}}}

\newcommand{\cU}{\mathcal{U}}
\newcommand{\cQ}{\mathcal{Q}}

\newcommand{\sF}{\mathscr{F}}
\newcommand{\sG}{\mathscr{G}}
\newcommand{\sH}{\mathscr{H}}

\newcommand{\tF}{{\rm T}\sF}
\newcommand{\tG}{{\rm T}\sG}
\newcommand{\tH}{{\rm T}\sH}
\newcommand{\nF}{{\rm N}\sF}
\newcommand{\nG}{{\rm N}\sG}

\newcommand{\Fol}{\mathbb{F}{\rm ol}}

\newcommand{\tX}{{\rm T}X}
\newcommand{\omx}[1]{\Omega_{X}^{#1}}
\newcommand{\lra}{\longrightarrow}
\DeclarePairedDelimiter{\ceil}{\lceil}{\rceil}

\DeclareMathOperator{\Flag}{Flag}
\DeclareMathOperator{\Log}{Log}
\newcommand{\OOP}{\mathbb{OP}}

\newcommand{\dlra}{%
  \mathrel{%
    \tikz[baseline=-0.5ex] \draw[dashed,line width=0.8pt,->] (0,0) -- (2em,0);
  }%
}

\setlist[enumerate]{label=\roman*)} 

\makeatletter 
\def\paragraph{\@startsection{paragraph}{4}%
  \z@\z@{-\fontdimen2\font}%
  {\normalfont\bfseries}}
\makeatother

\theoremstyle{plain}
\newtheorem{theorem}{Theorem}[section]
\newtheorem{proposition}[theorem]{Proposition}
\newtheorem{corollary}[theorem]{Corollary}
\newtheorem{lemma}[theorem]{Lemma}
\theoremstyle{definition}

\newtheorem{remark}[theorem]{Remark}

\newtheorem{innerthmintro}{Theorem}
\newenvironment{thmintro}[1]
{\renewcommand\theinnerthmintro{#1}\innerthmintro}
{\endinnerthmintro}

\author[V. Benedetti]{Vladimiro Benedetti}
\address[V. Benedetti]{Universit\'e C\^ote d'Azur, CNRS, Laboratoire J.-A. Dieudonn\'e, Parc Valrose, F-06108 Nice Cedex 2, {\sc France}}
\email{vladimiro.benedetti@univ-cotedazur.fr}

\author[C. Kuster]{Crislaine Kuster}
\address[C. Kuster]{Yau Mathematical Sciences Center \\  Tsinghua University \\ Haidian District\\  Beijing \\Postcode 100084\\  China}
\email{kusterc10@tsinghua.edu.cn \\ crislainekuster@gmail.com }

\author[A. Muniz]{Alan Muniz}
\address[A. Muniz]{Departamento de Matem\'atica \\ Centro de Ci\^encias Exatas e da Natureza \\ Universidade Federal de Pernambuco \\ Recife - PE, CEP 50740-560, Brasil}
\email{alan.nmuniz@ufpe.br}

\title{Minimal-degree foliations on cominuscule Grassmannians}

\date{April 2026}



\begin{document}

\begin{abstract}
Given $X$ a cominuscule Grassmannian (or irreducible Hermitian symmetric space) and an integer $p$, we compute the minimum $l(p)$ such that $H^0(\Omega_X^p(l(p))) \neq 0$. This allows us to conclude that any codimension-one foliation of degree zero on a cominuscule Grassmannian is a pencil of hyperplanes, improving a result of the first and third authors with D. Faenzi. We also deduce the structure of codimension-one foliations of degree one. Finally, we provide families of examples of high codimensional foliations of minimal degree on classical Grassmannians, Lagrangian Grassmannians, Spinor varieties, and the Cayley plane.
\end{abstract}

\maketitle

\section{Introduction}
A singular holomorphic foliation $\sF$ on a smooth complex manifold $X$ is a rule that associates (holomorphically) to a general point of $X$ a subspace of fixed codimension $p$ of its tangent space in such a way that through each point there is a unique immersed $p$-codimensional submanifold with prescribed tangent spaces. In the language of coherent sheaves, a foliation $\sF$ is the data of a (saturated) subsheaf $\tF \subset TX$, called the tangent sheaf of $\sF$, involutive under the Lie bracket $[\tF,\tF]\subset \tF$. The quotient $\nF = TX/\tF$ is called the normal sheaf. 

Each foliation is defined by a twisted differential form $\eta \in H^0(\Omega_X^p\otimes \det(\nF))$ uniquely determined up to a holomorphic scalar factor. If $X$ is projective, then for fixed integer $p$ and a line bundle $L$, the foliations of codimension $p$ satisfying $\det(\nF) \cong L$ define a quasi-projective variety 
\[
\Fol(X,p,L) \subset \PP H^0(\Omega^p_X\otimes L).
\]
Describing the variety $\Fol(X,p,L)$ is a central problem, and it is very challenging even in the simplest cases. If $X = \pn$ is the projective space, we identify $\Pic(X) = \ZZ$ and $\det(\nF) = \opn(d+p+1)$, where $d$ is called the degree of $\sF$. In this case, we write $\Fol(\pn, p,d) := \Fol(\pn,p,\opn(d+p+1))$. A satisfactory description of $\Fol(\pn,p,d)$ is only known for $d\leq 1$ and for $(p,d) = (1,2)$, see \cite{Jouanolou, CLN:components,LPT13}. Many partial results are known in other cases; see, for instance, \cite{CLP:deg3, CM-deg2}. 

Recently, attention has been given to low-degree foliations on other projective varieties. Here, if $\Pic(X) \cong \ZZ$, the degree is defined as in the projective space; and we also write $\Fol(X,p,d) = \Fol(X, p,\ox(d+p+1))$. In \cite{BFM}, the first and third authors and D. Faenzi used representation-theoretic techniques to address $\Fol(X, 1,0)$ when $X$ is a cominuscule Grassmannian, i.e., an irreducible Hermitian symmetric space. In \cite{Kuster}, the second author addresses the same problem for adjoint manifolds. For complete intersections, see \cite{ACM:FanoDist, figueira, folCompInt}. 

The purpose of this work is to study $\Fol(X,p, d(p))$ where $X$ is a cominuscule Grassmannian, and $d(p)$ is the minimum possible degree. In this direction, we explicitly describe $H^0(\Omega_X^p(l(p)))$, where $l(p)$ is the minimum possible twist. These cohomology computations follow from general theory scattered in the literature. We believe that compiling this information here is of independent interest; see Section \ref{S: diff forms}. In \cite[Theorem D]{BCM}, the authors prove that if $X$ is a cominuscule Grassmannian, then $H^q(\Omega_X^p(l)) \neq 0$ implies $l+q \geq p\,c_1(TX)/\dim(X)$. Our results show that, for $q=0$, such a bound is far from being sharp, but it is the best one can achieve without carefully analyzing the peculiarities of each type of cominuscule Grassmannian. The case $q>0$ was later improved by Jie Liu in his thesis \cite{liu-thesis}.

Combining this cohomological knowledge with techniques from the theory of holomorphic foliations, we proved some result about the spaces of foliations on cominuscule Grassmannians. Our first main result improves \cite[Theorem B]{BFM}; see Theorem \ref{thm:deg-zero}.

\begin{thmintro}{A}\label{mthm:A}
Let $X \hookrightarrow \PP(V^\vee)$ be a minimally embedded cominuscule Grassmannian. Then every codimension-one foliation of degree zero on $X$ is given by a pencil of hyperplane sections. Hence, set-theoretically, 
\[
\Fol(X,1,0) \cong G(2,V).
\]
\end{thmintro}

In \cite{Jouanolou}, Jouanolou proved that $\Fol(\pn,1,1)$ has two irreducible components, whose general elements are either a pencil of quadrics with a double hyperplane or a pullback of a degree-one foliation on $\p2$ via a linear projection. The same holds for hypersurfaces except for the 3-dimensional quadric $Q^3$, see \cite[Theorem B]{folCompInt}. The space $\Fol(Q^3,1,1)$ has three components, two as above and one whose general member is a foliation given by the action of a noncommutative 2-dimensional Lie algebra; see \cite{LPT13}. We provide the following structure theorem for degree-one codimension-one foliations on cominuscule Grassmannians; see Theorem \ref{thm:deg-one}. 

\begin{thmintro}{B}\label{mthm:B}
Let $X$ be a cominuscule Grassmannian, except $IG(n,2n)$, $n\geq 2$. Let $\sF$ be a codimension-one foliation of degree one on $X$. Then either
\begin{enumerate}
    \item $\sF$ is logarithmic of type $(1,2)$, i.e., there exist $f^2,g\in H^0(\ox(2))$ such that $\sF$ is given by the fibers of $(f^2:g) \colon X \dashrightarrow \p1$; or
    \item $\sF = \phi^*\sG$ is the pullback, via a rational map $\phi\colon X \dashrightarrow S$, of a foliation $\sG$ on a surface $S$. Moreover, the fibers of $\phi$ are rationally connected and define a degree-zero foliation with semistable tangent sheaf. 
\end{enumerate}    
\end{thmintro}

This result does not extend to Lagrangian Grassmannians $IG(n,2n)$. For $n\geq 3$, there are projections $\pi \colon IG(n,2n) \dashrightarrow IG(2,4) = Q^3$ (see proof of Theorem \ref{thm:min-examples-orthosympl}) such that for $\sG \in \Fol(Q^3,1,1)$ transverse to $\pi$, the $\pi^*\sG$ has also degree one. Therefore, we deduce the following result.

\begin{thmintro}{C} \label{mthm:C}
For every $n\geq 2$,  $\Fol(IG(n,2n),1,1)$ has at least three irreducible components.
\end{thmintro}

In the last part, we provide examples of foliations of minimal degree on the classical Grassmannian, the Lagrangian Grassmannian, and the Spinor variety. In particular, we prove that the respective spaces of foliations are nonempty; see Theorem \ref{thm_min_rect}, Remark \ref{rem:rect-part}, and Theorem \ref{thm:min-examples-orthosympl}.

\begin{thmintro}{D}\label{mthm:D}
Let $d(p)=l(p)-1-p$ be the minimal degree, where $l(p)$ denotes the minimal possible twist for the space of $p$-forms. The space of minimal-degree foliations $\Fol(X,p,d(p))$ is nonempty in the following cases:
\begin{itemize}
    \item $X= G(k,n)$, $n\geq 2k$, $l(p)^2-4p$ is a perfect square and $l(p)=d+e$ where $d\leq n-k$, $e\leq k$ are the (integer) solutions of $x^2-l(p)x+p=0$.
    \item $X = IG(n,2n)$ or $X = OG(n,2n)$ and $2p = a(a+1)$ with $a\leq n-2$.
\end{itemize}
\end{thmintro}

Beware that $l(p)$ depends on $p$ but also on the ambient variety on which we are constructing $p$-forms (so, the value of $l(p)$ depends on whether we are in the case of ordinary, symplectic or orthogonal Grassmannians). Notice moreover that, for any variety $X$, $\Fol(X,p,d)$ may be empty even though $H^0(X,\Omega_X^p(d+p+1))\neq 0$; the most striking example is the case when $X$ is an adjoint variety, $p=1$ and $d$ is minimal, in which case $H^0(X,\Omega_X^1(1))\cong \CC\omega$. In such a situation $\omega$ is the contact 1-form on $X$ and thus it is ``highly'' non-integrable.  Finally, one interesting open question we do not address in the present paper is whether the foliations produced in our examples fill irreducible components of the respective spaces of foliations. \medskip

This work is organized as follows. In section \ref{sec:prelim} we recall and prove basic facts about cominuscule Grassmannians and foliations. In section \ref{S: diff forms} we describe the exterior powers of the cotangent bundle of a cominuscule Grassmannian and compute the minimum twist $l(p)$ to have $p$-forms. Then, in Section \ref{S:cod-one foliations}, we prove Theorem \ref{mthm:A} and Theorem \ref{mthm:B}; see Theorem \ref{thm:deg-zero} and Theorem \ref{thm:deg-one}. Theorem \ref{mthm:C} is proven in Section \Ref{S: Examples}, see Theorem \ref{thm:min-examples-orthosympl} and Theorem \ref{T: IG(n,2n) components}.   Finally, in Section \ref{S: Examples}, we present examples of minimal-degree foliations and prove Theorem \ref{mthm:D} (see Theorems \ref{thm_min_rect} and \ref{thm:min-examples-orthosympl}). We also provide an example of codimension-8 foliation on the Cayley plane $\OOP^2$, see Proposition \ref{prop:fol-Cayley}.

{
\subsection*{Acknowledgements} 
We thank Peter Belmans and Jie Liu for their helpful comments on a previous version. Part of this work was carried out during the BRIDGES meeting Gauge theory, extremal structures, and stability, held in June 2024 at the Institut d’Études Scientifiques de Cargèse (Corsica). We thank the organizers for the opportunity to meet and discuss the problems addressed in this paper. Kuster also thanks IMPA (Brasil) for their support during the initial stages of this work.
}

\section{Preliminaries}\label{sec:prelim}
\subsection{Cominuscule Grassmannians}

As already shown in \cite{BFM}, among homogeneous varieties, cominuscule varieties are particularly suited for studying foliations. Any cominuscule variety, or (from the work of E. Cartan) Hermitian symmetric variety, decomposes as a product $X = X_1\times \cdots \times X_k$ where each $X_j$ is one of the irreducible (i.e., with Picard rank equal to one) cominuscule varieties appearing in Table \ref{tab:comin}: we call these irreducible factors \emph{cominuscule Grassmannians}. To study foliations, an important feature of a cominuscule variety $X$, which by the way characterizes the cominuscule property, is that its tangent bundle $TX$ is a completely reducible homogeneous bundle (\cite[Corolary 36]{BS:Hochschild}); for irreducible ones, the tangent bundle is even an irreducible homogeneous bundle. 

In the following, we will denote by $E_\lambda$ the irreducible vector bundle on a homogeneous variety $G/P$ associated to the weight $\lambda$ (see \cite[Section 3]{BFM} for more details). For a semisimple group $G$ of rank $r$, the fundamental weights will be denoted by $\lambda_1,\dots,\lambda_r$. Furthermore, we will denote by $\cU$ (respectively $\cQ$) the tautological rank $k$ subbundle (resp. the tautological rank $n-k$ quotient bundle) on a Grassmannian $G(k,n)$; finally $\cU_X^{\perp}$ will denote the rank $n-k$ orthogonal (w.r.t. the standard quadratic form) of the tautological bundle on $X=OG(k,n)$.

\begin{table}[ht]
    \centering
    \renewcommand{\arraystretch}{1.3}
    \begin{tabular}{|c|c|m{4.5cm}|c|c|c|}\hline
      Type   & diagram & description & dimension & $c_1(\tX)$ & $\Omega_X^1$ \\ \hline
        $A_r$    & \dynkin[scale=1.2] A{*.*x*.*} & Grassmannian $G(k,r)$ & $k(r-k)$  & $r$ & $\cU\otimes \cQ^\vee$ \\ \hline
        $B_r$    & \dynkin[parabolic=1,scale=1.2] B{} & Quadric $Q^{2r-1}$ & $2r-1$ & $2r-1$ & $\cU_X\otimes \cU_X^\perp/\cU_X $\\ \hline
        $C_r$    & \dynkin[parabolic=16,scale=1.2] C{}  & Lagrangian Grassmannian $IG(r,2r)$  & $ \dfrac{r(r+1)}{2} $ & $r+1$ & $S^2\cQ_X^\vee$ \\ \hline
        $D_r$    &  \dynkin[parabolic=1,scale=1.2] D{} & Quadric $Q^{2r-2}$ & $2r-2$ & $2r-2$ & $\cU_X\otimes \cU_X^\perp/\cU_X$ \\ \hline
        $D_r$    &  \dynkin[parabolic=32,scale=1.2] D{} & Spinor variety $OG(r,2r)$ & $\dfrac{r(r-1)}{2} $ & $2r-2$ & $\bigwedge^2 \cQ^\vee_X$ \\ \hline
        $E_6$    & \dynkin[parabolic = 1,scale=1.2] E6 & The Cayley plane $\mathbb{OP}^2$ & $ 16$ & $12$ & $E_{-2\lambda_1+\lambda_3}$ \\ \hline
        $E_7$    & \dynkin[parabolic = 64,scale=1.2] E7 & The Freudenthal variety & $27$ & $18$ & $E_{-2\lambda_7+\lambda_6}$ \\ \hline 
    \end{tabular}
    \medskip
    \caption{Cominuscule Grassmannians up to isomorphism.}
    \label{tab:comin}

\end{table}

We will be concerned with the exterior powers of the cotangent bundle $\Omega_X^1$, which is irreducible when $X$ is a cominuscule Grassmannian. The following technical lemma will be very useful.

\begin{lemma}
\label{lem:compute a}
Let $E_\lambda$ be an irreducible rank-$r$ vector bundle on $X = G/P$, for $P$ maximal parabolic corresponding to the root $\alpha_k$. Let $\mu$ be a partition, and suppose that the Schur bundle $\Gamma^\mu E_\lambda$ decomposes as $\Gamma^\mu E_\lambda = \bigoplus_{i=1}^k E_{\rho_i + a_i\lambda_k}$, for some weights $\rho_i$ and integers $a_i$. Then for each factor $E_{\rho_j + a_j\lambda_k}$ the $a_j$ are computed in terms of $\rho_j$ as follows
\begin{equation}\label{eq:weights-plethysm}
    a_j = \frac{\langle |\mu|\lambda - \rho_j \,, \lambda_k \rangle}{\langle \lambda_k \,, \lambda_k \rangle}
\end{equation}
where the pairing $\langle \cdot \,, \cdot \rangle$ is the one induced by the Killing form.
\end{lemma}

\begin{proof}
Since $E_\lambda$ is irreducible, it carries a (fiberwise) $\CC^*$-action of weight $c_1(E_\lambda)/\rk E_\lambda$ (meaning \emph{every} vector in each fiber has this weight), so the weight on $\det E_\lambda$ is $c_1(E_\lambda)$. Thus, each direct summand of $\Gamma^\mu E_\lambda$ must carry a $\CC^*$-action of weight $|\mu|c_1(E_\lambda)/\rk E_\lambda$. Taking the determinant of each summand yields
\[
\rk E_{\rho_i} \frac{|\mu|c_1(E_\lambda)}{\rk E_\lambda}  = c_1(E_{\rho_i + a_i\lambda_k}) = c_1(E_{\rho_i}) + a_i\rk E_{\rho_i}.
\]
Therefore, 
\[
a_i = \frac{|\mu|c_1(E_\lambda)}{\rk E_\lambda} - \frac{c_1(E_{\rho_i})}{\rk E_{\rho_i}}.
\]
To conclude the proof, by \cite[p.56]{ottaviani} (see also \cite[Lemma 2.4.1]{VladThesis}), one computes, for any weight $\gamma$,
\(
\langle \lambda_k\,, \lambda_k \rangle c_1(E_\gamma) = \rk  E_{\gamma} \langle \gamma \,, \lambda_k \rangle.
\)
Substituting into the previous equation yields the expression in the statement.
\end{proof}

To determine the $\rho_i$ that decompose $\Gamma^\mu E_\lambda$, we proceed as follows. The Dynkin diagram of the semisimple part of $P$ is obtained from the diagram of $G$ by removing the $k$-th simple root. If the expansion of $\lambda$ in terms of fundamental weights is $\lambda = \sum x_i\lambda_i$, define $\widehat{\lambda} = \sum_{i\neq k} x_i\lambda_i$, which we will regard as a weight for the modified diagram. The fiber of $E_\lambda$ at a $P$-invariant point affords the representation $V_{\widehat{\lambda}}$. Computing $\Gamma^\mu V_{\widehat{\lambda}} = \bigoplus_i V_{\rho_i}$, with respect to the modified Dynkin diagram, yields $\Gamma^\mu E_\lambda = \bigoplus_i E_{\rho_i + a_i\lambda_k}$. In general, the plethysm $\Gamma^\mu V_{\widehat{\lambda}}$ is challenging to compute. For specific cases, we use well-known formulas or the software Lie \cite{lie}.

To compute the pairings effectively, one may expand the weights in terms of the fundamental weights and use the inverse Cartan matrices. We refer to the last chapter of \cite{OniVinberg}, where this information is compiled.

\subsection{Holomorphic Foliations} \label{S: foliations}
Let $X$ be a complex projective manifold of dimension $n$. A holomorphic distribution on $X$ is given by the data of an exact sequence
\[
\sF : 0 \longrightarrow \tF \longrightarrow \tX \stackrel{\omega_\sF}{\longrightarrow} \nF \longrightarrow 0
\]
where $\tF$, called the tangent sheaf of $\sF$, is a saturated subsheaf of the tangent bundle $\tX$, i.e., $\nF$ is a torsion-free sheaf, called the normal sheaf of $\sF$.
The map on the right is defined by a twisted differential $p$-form $\omega_\sF \in H^0(\Omega_X^p \otimes \det \nF)$, where $p = \rk \nF$ is the codimension of $\sF$. The vanishing locus of $\omega_\sF$ is called the singular scheme of $\sF$ and denoted by $\sing(\sF)$. It coincides, set-theoretically, with the locus where $\nF$ is not locally free; in particular, $\codim \sing(\sF) \geq 2$. 

A distribution $\sF$ is called a foliation if, for each point $x\in X\setminus \sing(\sF)$, there exists a unique smooth immersed analytic subvariety $\mathcal{L} \ni x$ whose tangent space at $y \in \mathcal{L}$ coincides with $\tF(y)$; this is called a leaf of $\sF$. Due to Frobenius's Theorem, $\sF$ is a foliation if and only if $[\tF, \tF] \subset \tF$. 

On the other hand, given a line bundle $L$ and a twisted $p$-form $\omega\in H^0(\omx{p}\otimes L)$, it defines a distribution if the kernel of the contraction map $\tX \to \omx{p-1}\otimes L$, $v\mapsto \iota_v\omega$, has rank $n-p$. This is equivalent to asking that for every point $x$ not in the zero locus of $\omega$, there exists an open neighborhood $U\ni x$ and local $1$-forms $\alpha_1, \dots, \alpha_p$ such that $\omega|_U = \alpha_1\wedge \dots \wedge \alpha_p$; such differential forms are called locally decomposable off the singular set (LDS, for short). Equivalently, via de Rham's division lemma, we have Pl\"ucker’s decomposability condition: for every (local) section $\xi \in \extp^{p-1}\tX(U)$
\begin{equation}\label{eq:pluck}
    (\iota_\xi \omega)\wedge \omega = 0.
\end{equation}
Moreover, Frobenius's integrability condition translates to $\alpha_j \wedge d\omega = 0$, or 
\begin{equation}\label{eq:integra}
    (\iota_\xi \omega)\wedge d\omega = 0.
\end{equation}

A foliation $\sF$ is algebraically integrable if a general leaf is an algebraic variety, i.e., open in its Zariski closure. In this case, there is a dominant rational map $f \colon X \dashrightarrow Y$ with irreducible general fibers, such that the leaves of $\sF$ are open subsets of the fibers of $f$. For details, see \cite[\S3]{araujo-druel:fano-foliations}. The following observation will be useful later. 

\begin{lemma}\label{lem:normal-fibration}
Let $\varphi \colon X \dashrightarrow Y$ be a dominant rational map between projective manifolds with Picard groups isomorphic to $\ZZ$, such that $\varphi^*\OO_Y(1) = \OO_X(1)$. If $\sF$ is the foliation given by the fibers of $\varphi$, then 
\[
c_1(\nF) = c_1(TY) - \delta
\]
where $\delta \geq 0$ is the degree of the ramification divisor of $\varphi$. Moreover, if $y = \dim Y$ and $c_1(TY) = \min\{\, l \mid H^0(\Omega^y_X(l))\neq 0 \,\}$, then $\delta =0$.
\end{lemma}

\begin{proof}
Let $\eta$ be a global rational $y$-form on $Y$ so that the canonical divisor is $K_Y = {\rm div}(\eta) =  (\eta)_0 - (\eta)_\infty$. Then $\varphi^*\eta$ is a global closed rational $y$-form on $X$ defining $\sF$. One then has (at the level of divisor classes)
\[
c_1(\nF) = (\varphi^*\eta)_\infty - (\varphi^*\eta)_0 = \varphi^* ( \eta_\infty -\eta_0) -\Delta = \varphi^*c_1(TY) - \Delta,
\]
where $\Delta$ is the ramification divisor of $\varphi$. Since $\Delta$ is effective, the second part follows.
\end{proof}

\medskip
\paragraph{Space of foliations} Given $X$ a complex projective manifold with $\Pic(X) \cong \ZZ$, and $q,d \in \ZZ$, the \emph{space of foliations} of codimension $q$ and degree $d$ is the quasi-projective variety
\[
\Fol(X,q,d) = \{\, [\omega] \in \PP H^0(\Omega_X^q(d+q+1)) \mid \omega \text{ is integrable and } \codim \sing \omega \geq 2 \,\}.
\]

The study of $\Fol(X,q,d)$ is a central problem in global holomorphic foliation theory. Usually, one fixes $q$ and $d$ and describes the foliations corresponding to a general element of each irreducible component of $\Fol(X,q,d)$. For $X = \pn$, a complete description is only known for $d\leq 1$, or $(d,q) = (2,1)$, but partial results abound. We refer the reader to \cite{CM-deg2} and references therein. For $X$ a homogeneous variety and $(q,d) = (1,0)$ see \cite{BFM,Kuster}. For $X$ a complete intersection, $d\leq 1$ and $q=1$, see \cite{folCompInt}. One result that is recurrent in the literature is that for $q=1$ (e.g., for $X=\pn $ or $X$ a complete intersection) degree-0 foliations are defined by pencils of hyperplane sections:
\[
\Fol(X,1,0) \cong G(2, H^0(\ox(1))).
\]
In Section \ref{S:cod-one foliations}, we will show that this holds for every cominuscule Grassmannian, complementing the main result of \cite{BFM}. We will also give a partial result for $q=d=1$. In Section \ref{S: Examples}, we will provide examples of minimal degree foliations of high codimension.

\medskip
\paragraph{Slope stability}
Let $X$ be a complex projective manifold with $\Pic(X) \cong \ZZ$. The \emph{slope} of a torsion-free coherent sheaf $F$ on $X$ is $\mu(F) = \frac{c_1(F)}{\rk F}$. One says that $F$ is \emph{semistable} if for every $E\subset F$ with $\rk E < \rk F$, one has $\mu(E) \leq \mu(F)$. If, moreover, there is no $E$ such that equality holds, one says that $F$ is \emph{stable}. Given a distribution (or foliation) $\sF$ on $X$, the stability of $\tF$ controls both the possible subfoliations and the algebraic integrability, especially when $\mu(\tF) > 0$. We refer the reader to \cite{Carol-ICM, CP} and references therein for results in this direction. We will need the following lemma.

\begin{lemma}\label{lem:integ-stab}
Let $X$ be a complex projective manifold with $\Pic(X) \cong \ZZ$, and let $\sF$ be a foliation on $X$ such that $\mu(\tF) >0$. If $q = \codim \sF$ and $c = c_1(\nF)$, and $H^0(\Omega_X^p(c))$ has no integrable forms for $q < p  < \dim X$, then $\sF$ is algebraically integrable with rationally connected leaves. Moreover, if $q=1$ and $X$ is simply connected, then $\sF$ is given by rational first integral of the form $(f^a,g^b) \colon X \dashrightarrow \p1$ with $f\in H^0(\ox(l))$, $g\in H^0(\ox(m))$ satisfying $\gcd(a,b) = 1$, $al=bm$, and $c = l+m$.
\end{lemma}

\begin{proof}
By \cite[Prop. 7.5]{araujo-druel:fano-foliations}, either $\tF$ is semistable and $\sF$ is algebraically integrable with rationally connected leaves, or there exists a subfoliation $\sG \subset \sF$ such that $c_1(\tG) \geq c_1(\nF)$, equivalently $c_1(\nG) \leq c_1(\nF) = c$. Such a foliation $\sG$ would be defined by an integrable form $\eta \in H^0(\Omega^p_X(c_1(\nG)))$ for some $p>q$. But $H^0(\Omega^p_X(c_1(\nG)))$ has no integrable form by hypothesis. Indeed, multiplying by an element of $H^0(\ox(c-c_1(\nG)))$ would produce a nonzero integrable form in $H^0(\Omega^p_X(c))$. 

Finally, if $q=1$ and $X$ is simply connected, then \cite[Prop. 3.5]{LPT13} implies that $\sF$ is given by a rational first integral as in the statement.
\end{proof}

\section{Differential forms on cominuscule Grassmannians} \label{S: diff forms}
In the papers \cite{Snow-quad,Snow-comin}, Dennis Snow provided a series of results characterizing the cohomology of twisted differential forms $H^p(\Omega^q_X(k))$ for $X$ a cominuscule Grassmannian, also called \emph{irreducible Hermitian symmetric space}. The strategy is to describe the weights of $\Omega^q_X(k)$ and apply the Bott-Borel-Weil Theorem. The description of the weights follows from Kostant's work \cite{Kos}. Here, we reproduce some of Snow's results in a more straightforward and perhaps concrete way, by a case-by-case analysis; even though we focus only on global sections, our results allow to identify explicitly the representations corresponding to the space of twisted global differential forms. Our main interest is in computing the smallest $k$ for which $H^0(\Omega^p_X(k)) \neq 0$. For the classical Grassmannians, we follow \cite{KU-fano}. 

\medskip
\paragraph{Classical Grassmannians}
Consider the classical Grassmannian $X = G(k,V)$ of $k$-dimensional subspaces of $V$, $\dim V = n$, with tautological sequence
\[
0 \longrightarrow \U \longrightarrow \OO_X\otimes V \longrightarrow \Q \longrightarrow 0
\]
where $\U$ is the tautological rank $k$ bundle. Following \cite{KU-fano}, we write the fundamental weights of $\fsl(V)$ as $\lambda_i = e_1 + \cdots +e_i$, for $1\leq i \leq n-1$,  where $\epsilon = \{ e_1, \dots , e_{n}\}$ a basis of $\RR^{n}$. Moreover, $\lambda_n = e_1 + \cdots + e_n$ is identified with $0$. For a weight $\beta$, one writes it as $\beta = \sum \beta_j e_j=\beta_1\lambda_1+\beta_2(\lambda_2-\lambda_1)+\cdots+\beta_n(\lambda_n-\lambda_{n-1})$, or
\[
\beta = (\beta_1, \dots, \beta_k; \beta_{k+1}, \dots, \beta_n).
\]
In this notation, $\beta$ is dominant if and only if $\beta_{j} \geq \beta_{j+1}$ for $j=1, \dots, n-1$. If $E_\beta$ is the corresponding vector bundle, the dual $E_{\beta}^\vee$ corresponds to the weight
\[
(-\beta_k , \dots, -\beta_1; -\beta_n, \dots, -\beta_{k+1})
\]
In particular, $\U^\vee$ is associated with the weight $\lambda_1$ and $\Q$ with $\lambda_{n-1}$. Then, $TX = \U^\vee \otimes \Q$ corresponds to 
\[
\lambda_1 + \lambda_{n-1} = (2,1,\dots,1;1,\dots,1,0) = (1, 0 , \dots, 0,-1)
\]
and $\Omega^1_X = \Q^\vee \otimes \U$ corresponds to $(0, \dots, 0,-1;1,0,\dots,0) = \lambda_{k-1}-2\lambda_k+\lambda_{k+1}$, According to \cite[p.60]{weyman:tract},
\[
\Omega^p_{G(k,V)} = \extp^p(\mathcal{Q}^\vee \otimes \mathcal{U}) = \bigoplus_{|\mu| = p}  \Gamma^{\mu}\mathcal{U} \otimes \Gamma^{\mu'} \mathcal{Q}^\vee,
\]
where $\Gamma^\mu$ is the Schur functor associated with the partition $\mu$ and $\mu'$ is the dual partition of $\mu$. 

\begin{remark}\label{rem:notation}
In Weyman's notation $\Gamma^\mu = L_{\mu'}$. In particular, $\Gamma^{(p)}E = S^pE$ is the symmetric power.    
\end{remark}

The summand corresponding to the partition $\mu$ has weight
\[
\beta_\mu = (-\mu_k, \dots, -\mu_1; \mu_1', \dots, \mu_{n-k}').
\]
In particular, $\mu$ has at most $k$ parts, i.e., $\mu_1' \leq k$ and, similarly, $\mu_1 \leq n-k$. Graphically, the Young diagram of $\mu$ fits within a $(n-k)\times k$ rectangle.

\begin{lemma}\label{lem:h0-min}
$H^0(\Gamma^{\mu}\mathcal{U} \otimes \Gamma^{\mu'} \mathcal{Q}^\vee(l)) \neq 0$ if and only if $\mu_1+\mu'_1 \leq l$. In particular, 
\[
\mu_1 + \mu'_1 = \min\{\, j \in \ZZ \mid H^0(\Gamma^{\mu}\mathcal{U} \otimes \Gamma^{\mu'} \mathcal{Q}^\vee(j)) \neq 0\,\}.
\]
\end{lemma}

\begin{proof}
Using the notation above, $\Gamma^{\mu}\mathcal{U} \otimes \Gamma^{\mu'} \mathcal{Q}^\vee(l)$ corresponds to the weight
\[
\beta_\mu + l \lambda_k = (l-\mu_k, \dots, l-\mu_1; \mu_1', \dots, \mu_{n-k}').
\]
By the Bott-Borel-Weyl Theorem, $H^0(\Gamma^{\mu}\mathcal{U} \otimes \Gamma^{\mu'} \mathcal{Q}^\vee(l)) \neq 0$ if and only if $\beta_\mu + l\, \lambda_k$ is dominant. This is equivalent to $l-\mu_1 \geq \mu_1'$, which concludes the proof.
\end{proof}

As a consequence of this lemma, we compute the minimum twist for which $\Omega_{G(k,n)}^p$ has global sections.

\begin{proposition}\label{prop:min-deg-grass}
Let $X = G(k,n)$ be a classical Grassmannian, where $n\geq 2k$, and let $l(p) = \min \{\, l \in \ZZ \mid H^0(\Omega_X^p(l)) \neq 0 \,\}$. 
Then
\[
l(p) = \begin{cases}\displaystyle
     k + \ceil*{\frac{p}{k}}, & k^2 < p < k(n-k) \\[2mm]
    \displaystyle \ceil*{2\sqrt{p}}, & p \leq k^2
\end{cases}.
\]
\end{proposition}

\begin{proof}
This follows from Lemma \ref{lem:h0-min}. Note that for a given $p$, we need to find the partitions $\mu$ of $p$ that minimize $\mu_1+\mu'_1$. Drawing Young diagrams, one sees that $\mu$ is a diagram of area $p$ inside a $\mu_1'\times \mu_1$ rectangle. To find the optimal rectangles, note that if $\mu_1'$ is given, the smallest $\mu_1$ is $\mu_1 = \ceil*{\frac{p}{\mu_1'}}$. Hence, we want to minimize $f(x) = x + \ceil*{\frac{p}{x}}$ for $1\leq x\leq k$. 

Suppose that $p> k^2$. We will show that the minimum is attained at $\mu_1 = k$. For $x\geq 2$, one has $\ceil*{\frac{p}{x}} \leq \ceil*{\frac{p}{x-1}}$. If equality holds, $\ceil*{\frac{p}{x}} = \ceil*{\frac{p}{x-1}} = u$, then $u-1 < \frac{p}{x} < \frac{p}{x-1} \leq u$. Rewriting, one gets $(u-1)x < p < u(x-1)$ which implies that $x > u$. On the other hand, since $xu \geq p$, one has that $x > u$ implies $x > \sqrt{p}$. Therefore, for $x \leq k < \sqrt{p}$, $f(x) - f(x-1) = 1+ \ceil*{\frac{p}{x}} -\ceil*{\frac{p}{x-1}} \leq 0$. Thus, $f(x)$ is nonincreasing for integers $x\leq k < \sqrt{p}$, and the minimum is attained at $x=k$.

Next, suppose that $p \leq k^2$. Instead of minimizing $f(x)$ as in the previous case, we observe that for any $a,b$ such that $ab\geq p$, one has $a+b \geq 2\sqrt{ab} \geq 2\sqrt{p}$. Fix $m_p = \ceil*{2\sqrt{p}}$. We will show that $l(p) = m_p$. For the moment, what we have said implies that $l(p)\geq m_p$.  Consider $g(x) = x^2 -m_px+p$ and note that for any partition $\mu$, satisfying $\mu_1+\mu_1' = m_p$, $g(\mu_1) \leq 0$. Conversely, if $g(x) \leq 0$ for some integer $x$, either $x \leq \frac{m_p}{2} \leq k$ or $m_p - x \leq k$. Let us show that $g(x) \leq  0 $ for some integer $x$. Consider the discriminant $\Delta = m_p^2 - 4p$; note that $\Delta \geq 0$. If $\Delta \geq 1$, then there must be an integer between the roots of $g$, and we are done. But if $\Delta <1$ then $\Delta = 0$ and $4p = m_p^2$; then $p$ is a square, $m_p$ is even, and the only root of $g$ is $x = \frac{m_p}{2}\in \ZZ$. So $x$ exists, and either $x$ or $m_p-x$ is $\leq k$. Say $x\leq k$, then $\nu=((m_p-x)^x)$ is a rectangular partition satisfying $|\nu|=x(m_p-x)\geq p$; by removing some blocks from $\nu$ one obtains a partition $\mu$ such that $|\mu|=p$ and $\mu_1+\mu_1'\leq m_p$ thus showing that $l(p)\leq m_p$, thus concluding the proof.
\end{proof}

\begin{remark}\label{rem:unique-rectangles}
In Proposition \ref{prop:min-deg-grass}, although $l(p)$ is determined, the partitions, or even the minimal rectangles they fit into, may not be unique. For instance, if $k=3$ and $p=10$ one gets $l(p) = 3 + \ceil*{\frac{10}{3}} = 7$, but also $2 + \ceil*{\frac{10}{2}} = 7$. So $\mu_1 = 2,3$ give two minimal rectangles. We get the minimal partitions $(2^5)$ and $(3^3,1)$. For $k=3$ and $p=7$ one gets $l(p) = \ceil*{2\sqrt{7}} = 6$. Hence, the partitions are $(4,3)$, $(3^2,1)$, and $(3,2^2)$. The last two fit in a $3\times 3$ square. Nonetheless, the proof of the proposition gives a recipe for determining all the possible partitions for given $k$ and $p$. 
\end{remark}

\begin{remark}\label{rem:rect-part}
In the last section, we will construct foliations associated with a rectangular partition $\mu = (d^e)$, where $d\leq n-k$ and $e\leq k$. These partitions correspond to representations $V_{\beta_\mu + l\lambda_k}V^\vee \subset H^0(\Omega^p_{G(k,V)}(l))$, where $l = d+e$ and $p=de$. Imposing that $d+e = l(p)$, i.e., that the form has minimum degree, one has that $d$ and $e$ are solutions to the quadratic equation $x^2 - l(p)x+p = 0$. Therefore, such partitions occur only when $l(p)^2 - 4p = \delta^2$ is a perfect square. In this case, the rectangular partition $\mu$ is unique, unless $d\leq k$ so that $\mu' = (e^d)$ also occurs.
\end{remark}

\medskip
\paragraph{Lagrangian Grassmannians}
The Lagrangian Grassmannian $X= IG(n,2n)$ pa\-ra\-me\-te\-ri\-zes maximal isotropic subspaces of $\CC^{2n}$ with respect to a nondegenerate symplectic form. It is thus given as the zero locus of a global section of $\extp^2\Q$, where $\Q$ is the tautological quotient bundle of $G(n,2n)$. The cotangent bundle is $\Omega_X^1 = S^2\Q_X^\vee$, where \( \Q_X \) denotes the restriction of \( \Q \) to \( X \). By \cite[Proposition 2.3.9]{weyman:tract} and Remark \ref{rem:notation}, 
\[
\Omega_X^p = \extp^p S^2\Q_X^\vee = \bigoplus_{\mu \in Q_{1}(2p)} \Gamma^\mu \Q_X^\vee
\]
where $Q_{1}(2p)$ is the set of partitions of $2p$ that can be written as $(a_1, \dots, a_n | b_1, \dots, b_n)$ with $a_i = b_i +1$ in hook notation. In particular, $\mu_1' = \mu_1-1$.

\begin{lemma}\label{lem:h0-lagr}
Let $X= IG(n,2n)$, and let $\Q_X$ denote the restriction of $\Q$ to $X$. Then, for $\mu \in Q_{1}(2p)$, we have $H^0(\Gamma^\mu \Q_X^\vee(l)) \neq 0$ if and only if $l\geq \mu_1$.
\end{lemma}

\begin{proof}
Note that $\Q_X^\vee = E_{\lambda_{n-1}-\lambda_n}$. Since $X = C_n/P_n$ and the semisimple part of $P_n$ is of type $A_{n-1}$, one gets 
\[
\Gamma^\mu E_{\lambda_{n-1}-\lambda_n} = E_{\beta_\mu + a\lambda_n}, \quad \text{where} \quad \beta_{\mu} := \sum_{i=1}^{n-1}(\mu_i-\mu_{i+1})\lambda_{n-i}.
\]
By Lemma \ref{lem:compute a} and using that $\langle \lambda_j, \lambda_n \rangle = j$ (see \cite[p.296]{OniVinberg}), 
\begin{align*}
    a & = \frac{\langle |\mu|(\lambda_{n-1}-\lambda_n) -\beta_\mu, \lambda_n \rangle}{\langle \lambda_n, \lambda_n \rangle}= \frac{1}{n}\left(-|\mu| - \sum_{i=1}^{n-1}(\mu_i-\mu_{i+1})(n-i) \right)= \\
    & = \frac{1}{n}\left(-|\mu| - \sum_{i=1}^{n-1}\mu_i(n-i) +  \sum_{j=2}^{n}\mu_{j}(n-j+1) \right)= \frac{1}{n}\left(-\mu_1 - \mu_1(n-1) \right) = -\mu_1.
\end{align*}

We have that $\Gamma^\mu \Q_X^\vee(l) = E_{\gamma}$    where $\gamma = \beta_\mu +(l+a)\lambda_n$. Since $\mu_i\geq\mu_{i+1}$, $\gamma$ is dominant, hence $H^0(\Gamma^\mu \Q_X^\vee(l)) \neq 0$, if and only if $l\geq \mu_1$.
\end{proof}

\begin{proposition}\label{prop:min-deg-lagr}
 Let $X = IG(n,2n)$ be a Lagrangian Grassmannian, and let $l(p) = \min \{\, l \in \ZZ \mid H^0(\Omega_X^p(l)) \neq 0 \,\}$. Then
 \[
 l(p) = \ceil*{\sqrt{2p} + \frac{1}{2}}.
 \]
\end{proposition}

\begin{proof}
By the same reasoning as for the classical Grassmannians, we have Young diagrams of given area $2p$, and we need to minimize the perimeter of the rectangle containing them. The condition imposed by  $Q_{1}(2p)$ implies that $\mu_1' = \mu_1-1$. Thus $2\mu_1 - 1 \geq 2\sqrt{2p}$ and the minimum is $\mu_1 = \ceil*{\sqrt{2p} +\frac{1}{2}}$. 
\end{proof}

\begin{remark}\label{rem:rect-lagr}
Note that the partition $\mu$ is rectangular if and only if $2p = a(a+1)$, for $a = {\sqrt{2p} -\frac{1}{2}}$. In this case, $H^0(\Omega_X^p(l(p))) = V_{(a+1)\lambda_{n-a}}$. Indeed, a rectangular partition in $Q_1(2p)$ is automatically of the form $(a+1) \times a$; moreover, for such $p$, this is the only partition in $Q_1(p)$ contained in the rectangle of minimal perimeter.
\end{remark}

\medskip
\paragraph{Spinor varieties}
The Spinor variety $X = OG(n,2n)$ parameterizes $n$-dimensional subspaces of $\CC^{2n}$ isotropic to a symmetric form. There are several similarities to the previous case, arising from replacing the symplectic form with a symmetric one. The variety $X$ is defined as the zero locus of a global section of $\extp^2 \Q$, for $\Q$ the tautological quotient bundle of $G(n,2n)$. The cotangent bundle is $\Omega_X^1 = \extp^2\Q_X^\vee$, where $\Q_X$ is the restriction of $\Q$ to $X$. By \cite[Proposition 2.3.9]{weyman:tract} and Remark \ref{rem:notation}, 
\[
\Omega_X^p = \extp^p \bigwedge^2\Q_X^\vee = \bigoplus_{\mu \in Q_{-1}(2p)} \Gamma^\mu \Q_X^\vee
\]
where $Q_{-1}(2p)$ is the set of partitions of $2p$ that can be written as $(a_1, \dots, a_n | b_1, \dots, b_n)$ with $a_i = b_i -1$  in hook notation. In particular, $\mu_1' = \mu_1+1$. 

\begin{lemma}\label{lem:h0-spin}
Let $X = OG(n,2n)$, and let $\Q_X$ denote the restriction of $\Q$ to $X$. Then, for $\mu \in Q_{-1}(2p)$, we have $H^0(\Gamma^\mu \Q_X^\vee(l)) \neq 0$ if and only if $l\geq \mu_1+\mu_2$.
\end{lemma}

\begin{proof}
Note that $\Q_X^\vee = E_{\lambda_{n-1}-\lambda_n}$. Since $X = D_n/P_n$ and the semisimple part of $P_n$ is of type $A_{n-1}$,
\[
\Gamma^\mu E_{\lambda_{n-1}-\lambda_n} = E_{\beta_\mu + a\lambda_n}, \quad \text{where} \quad \beta_{\mu} := \sum_{i=1}^{n-1}(\mu_i-\mu_{i+1})\lambda_{n-i}.
\]
From \cite[p.296]{OniVinberg}, we have that $\langle \lambda_j, \lambda_n \rangle = \frac{j}{2}$ if $j\leq n-2$, $\langle \lambda_{n-1}, \lambda_n \rangle = \frac{n-2}{4}$, and $\langle \lambda_{n}, \lambda_n \rangle = \frac{n}{4}$. By Lemma \ref{lem:compute a},    
\begin{align*}
    a &= \frac{\langle |\mu|(\lambda_{n-1}-\lambda_n) -\beta_\mu, \lambda_n \rangle}{\langle \lambda_n, \lambda_n \rangle}=  \frac{1}{n}\left(-2|\mu| -(\mu_1-\mu_2)(n-2) -2\sum_{i=2}^{n-1}(\mu_i-\mu_{i+1})(n-i)\right) \\
    &=  -\mu_1-\mu_2.
\end{align*}
The result follows as in the symplectic case.
\end{proof}

\begin{proposition}\label{prop:min-deg-spinor}
 Let $X = OG(n,2n)$ be a Spinor variety, and let $l(p) = \min \{\, l \in \ZZ \mid H^0(\Omega_X^p(l)) \neq 0 \,\}$. Let $a = \ceil*{\sqrt{2p} - \frac{1}{2}}$ and write $2p = a(a+1)-2b$ for some $0 \leq b < a$. Then 
 \[
 l(p) = \begin{cases}
     2a,   & b \leq a-2 \text{ or } a=1\\
     2a-1, & b = a-1 >0
 \end{cases}.
 \]
\end{proposition}

\begin{proof}
The proof is completely analogous to Proposition \ref{prop:min-deg-lagr} as the partitions in $Q_{-1}(2p)$ are dual to those in $Q_{1}(2p)$. The diagrams are transposed along the main diagonal. The minimum $\mu_1$ is $\mu_1 = a:= \ceil*{\sqrt{2p} - \frac{1}{2}}$. To get $l(p) = \mu_1 + \mu_2$ we need to estimate $\mu_2$.

Note that $2p = a(a+1) - 2b$, with $0\leq b < a$. Hence, a minimal partition is obtained from a rectangle of size $(a+1) \times a$ by removing  $2b$ boxes, respecting $Q_{-1}(2p)$. If $a=1$ then $b=0$, hence $l(1) = 2$. Next suppose that $a\geq 2$. If $b \leq a-2$, then $\mu_2 = \mu_1$ for any partition. If $b=a-1$ then one can remove boxes forming a hook, resulting in $\mu = (a,(a-1)^{a-1},1)$, where $\mu_2 = \mu_1-1$.
\end{proof}

\begin{remark}\label{rem:rect-spin}
Note that the partition $\mu$ is rectangular if and only if $2p = a(a+1)$, for $a = {\sqrt{2p} -\frac{1}{2}}$. In this case, $H^0(\Omega_X^p(l(p))) = V_{a\lambda_{n-a-1}}$. Indeed, a rectangular partition in $Q_{-1}(2p)$ is automatically of the form $a \times (a+1)$; moreover, for such $p$ this is the only partition in $Q_{-1}(p)$ contained in the rectangle of minimal perimeter.
\end{remark}

\medskip
\paragraph{Cayley plane} Let $X = E_6/P_1$ be the Cayley plane. Then $\Omega_X^1 = E_{-2\lambda_1+\lambda_3}$. Taking exterior powers, one gets $\Omega_X^p = \bigoplus_{i=1}^3 E_{\beta_{p,i}}$ where $\beta_{p,i}$ are given in Table \ref{tab:cohE6},  To compute the entries of this table, first analyze the decompositions, noting that the semisimple part of $P_1$ is of type $D_5$. Then we apply Lemma \ref{lem:compute a}. We used the software Lie \cite{lie} to carry out these computations. 


\begin{table}[ht]
\centering
\small
\centering
\begin{tabular}{c*{3}{>{\centering\arraybackslash}p{3.2cm}}|c}
\toprule
$p\backslash i$ & 1 & 2 & 3 & $l(p)$\\
\midrule
1  & $-2\lambda_1+\lambda_3$ & -- & -- &$2$ \\
2  & $-3\lambda_1+\lambda_4$ & -- & -- & $3$ \\
3  & $-4\lambda_1+\lambda_2+\lambda_5$ & -- & --& $4$\\
4  & $-5\lambda_1+2\lambda_2+\lambda_6$ & $-5\lambda_1+2\lambda_5$ & -- &$5$\\
5  & $-6\lambda_1+3\lambda_2$ & $-6\lambda_1+\lambda_2+\lambda_5+\lambda_6$ & -- &$6$\\
6  & $-7\lambda_1+2\lambda_2+\lambda_5$ & $-7\lambda_1+\lambda_4+2\lambda_6$ & -- &$7$\\
7  & $-8\lambda_1+\lambda_2+\lambda_4+\lambda_6$ & $-8\lambda_1+\lambda_3+3\lambda_6$ & -- &$8$\\
8  & $-9\lambda_1+2\lambda_4$ & $-9\lambda_1+\lambda_2+\lambda_3+6\lambda_6$ & $-8\lambda_1+4\lambda_6$ &$8$\\
9  & $-10\lambda_1+\lambda_3+\lambda_4+\lambda_6$ & $-9\lambda_1+\lambda_2+3\lambda_6$ & -- &$9$\\
10 & $-10\lambda_1+\lambda_4+2\lambda_6$ & $-11\lambda_1+2\lambda_3+\lambda_5$ & -- &$10$\\
11 & $-12\lambda_1+3\lambda_3$ & $-11\lambda_1+\lambda_3+\lambda_5+\lambda_6$ & -- &$11$\\
12 & $-12\lambda_1+2\lambda_3+\lambda_6$ & $-11\lambda_1+2\lambda_5$ & -- & $11$\\
13 & $-12\lambda_1+\lambda_3+\lambda_5$ & -- & -- &$12$\\
14 & $-12\lambda_1+\lambda_4$ & -- & --&$12$ \\
15 & $-12\lambda_1+\lambda_2$ & -- & -- &$12$\\
\bottomrule
\end{tabular}

\caption{Weights $\beta_{p,i}$ decomposing $\Omega_X^p$, for $X$ the Cayley plane, and $l(p)= \min \{\,l \in \ZZ \mid H^0(\Omega_X^p(l))\neq 0 \,\}$. }
\label{tab:cohE6}
\end{table}


\medskip
\paragraph{Freudenthal variety} Let $X=E_7/P_7$ be the Freudenthal variety. Then $\Omega_X^1 = E_{-2\lambda_7+\lambda_6}$. As for the Cayley plane, we write $\Omega_X^p = \bigoplus_{i=1}^3 E_{\beta_{p,i}}$ where $\beta_{p,i}$ and determine $\beta_{p,i}$ explicitly in Table \ref{tab:cohE7}.

\begin{table}[ht]
\centering
\small
\centering
\begin{tabular}{c*{3}{>{\centering\arraybackslash}p{4cm}}|c}
\toprule
$p\backslash i$ & 1 & 2 & 3 & $l(p)$\\
\midrule
1  & $-2\lambda_7 +\lambda_6$ & -- & -- &$2$ \\
2  & $-3\lambda_7 +\lambda_5$ & -- & -- &$3$ \\
3  & $-4\lambda_7 +\lambda_4$ & -- & -- &$4$ \\
4  & $-5\lambda_7 +\lambda_2+\lambda_3$ & -- & -- &$5$ \\
5  & $-6\lambda_7+2\lambda_3$ & $-6\lambda_7+\lambda_1+2\lambda_2$ & -- & 6 \\
6  & $-7\lambda_7+3\lambda_2$ & $-7\lambda_7+\lambda_1+\lambda_2+\lambda_3$ & -- & 7 \\
7  & $-8\lambda_7+2\lambda_2+\lambda_3$ & $-8\lambda_7+2\lambda_1+\lambda_4$ & -- & 8 \\
8  & $-9\lambda_7+\lambda_1+\lambda_2+\lambda_4$ & $-9\lambda_7+3\lambda_1+\lambda_5$ & -- & 9 \\
9  & $-10\lambda_7+2\lambda_4$ & $-10\lambda_7+2\lambda_1+\lambda_2+\lambda_5$ & $-10\lambda_7+4\lambda_1+\lambda_6$ & 10 \\
10  & $-11\lambda_7+\lambda_1+\lambda_4+\lambda_5$ & $-11\lambda_7+3\lambda_1+\lambda_2+\lambda_6$ & $-10\lambda_7+5\lambda_1$ & 10 \\
11  & $-12\lambda_7+\lambda_3+2\lambda_5$ & $-12\lambda_7+2\lambda_1+\lambda_4+\lambda_6$ & $-11\lambda_7+4\lambda_1+\lambda_2$ & 11 \\
12  & $-13\lambda_7+3\lambda_5$ & $-13\lambda_7+\lambda_1+\lambda_3+\lambda_5+\lambda_6$ & $-12\lambda_7+3\lambda_1+\lambda_4$ & 12 \\
13  & $-14\lambda_7+2\lambda_3+2\lambda_6$ & $-14\lambda_7+\lambda_1+2\lambda_5+\lambda_6$ & $-13\lambda_7+2\lambda_1+\lambda_3+\lambda_5$ & 13 \\
14  & $-15\lambda_7+\lambda_3+\lambda_5+2\lambda_6$ & $-14\lambda_7+\lambda_1+2\lambda_3+\lambda_6$ & $-14\lambda_7+2\lambda_1+2\lambda_5$ & 14 \\
15  & $-16\lambda_7+\lambda_4+3\lambda_6$ & $-14\lambda_7+3\lambda_3$ & $-15\lambda_7+\lambda_1+\lambda_3+\lambda_5+\lambda_6$ & 14 \\
16  & $-15\lambda_7+2\lambda_3+\lambda_5$ & $-17\lambda_7+\lambda_2+4\lambda_6$ & $-16\lambda_7+\lambda_1+\lambda_4+2\lambda_6$ & 15 \\
17  & $-18\lambda_7+5\lambda_6$ & $-16\lambda_7+\lambda_3+\lambda_4+\lambda_6$ & $-17\lambda_7+\lambda_1+\lambda_2+3\lambda_6$ & 16 \\
18  & $-16\lambda_7+2\lambda_4$ & $-17\lambda_7+\lambda_2+\lambda_3+2\lambda_6$ & $-18\lambda_7+\lambda_1+4\lambda_6$ & 16 \\
19  & $-18\lambda_7+\lambda_3+3\lambda_6$ & $-17\lambda_7+\lambda_2+\lambda_4+\lambda_6$ & -- & 17 \\
20  & $-18\lambda_7+\lambda_4+2\lambda_6$ & $-17\lambda_7+2\lambda_2+\lambda_5$ & -- & 17 \\
21  & $-18\lambda_7+\lambda_2+\lambda_5+\lambda_6$ & $-17\lambda_7+3\lambda_2$ & -- & 17 \\
22  & $-18\lambda_7+2\lambda_5$ & $-18\lambda_7+2\lambda_2+\lambda_6$ & -- & 18 \\
23  & $-18\lambda_7+\lambda_2+\lambda_5$ & -- & -- & 18 \\
24  & $-18\lambda_7+\lambda_4$ & -- & -- & 18 \\
25  & $-18\lambda_7+\lambda_3$ & -- & -- & 18 \\
26  & $-18\lambda_7+\lambda_1$ & -- & -- & 18 \\
\bottomrule
\end{tabular}

\medskip
\caption{Weights $\beta_{p,i}$ decomposing $\Omega_X^p$, for $X$ the Freudenthal Variety, and $l(p)= \min \{\,l \in \ZZ \mid H^0(\Omega_X^p(l))\neq 0 \,\}$.}
\label{tab:cohE7}
\end{table}

\medskip
\paragraph{Quadrics} The case of the $n$-dimensional quadric $Q^n \subset \p{n+1}$ is easier when regarded as a hypersurface rather than a quotient of the orthogonal group $\SO(n+2)$. Considering the conormal sequence, one proves that $H^0(\Omega_{Q^n}^p(p)) = 0$, for every $1 \leq p \leq  n-1$, see \cite[Lemma 5.3]{ACM:FanoDist}. Moreover, $\Omega_{Q^n}^p(p+1)$ is globally generated, since it is a quotient of $\Omega_{\p{n+1}}^p(p+1)$. 

\medskip
The preceding results imply the following corollary.

\begin{corollary}\label{cor:nonvanishing} Let $X$ be a cominuscule Grassmannian. 
\begin{enumerate} 
\item If $H^0(\Omega_X^p(2)) \neq 0$ then $p=1$.  
\item If $H^0(\Omega_X^p(3)) \neq 0$ then $p\leq 2$, unless $X$ is a Lagrangian Grassmannian, and $p=3$. \end{enumerate} 
\end{corollary}

\section{Codimension-one foliations on cominuscule Grassmannians}\label{S:cod-one foliations}

In this section, we classify degree-zero codimension-one foliations on cominuscule Grassmannians and establish a structural theorem for degree-one codimension-one foliations on these varieties.

\medskip
\paragraph{Degree-zero foliations}
In \cite[Theorem B]{BFM}, it was proved that degree-0 foliations on certain cominuscule Grassmannians are necessarily pencils of hyperplane sections. More precisely, there is an isomorphism of schemes between the space of degree-0 foliations $\Fol (X,1,0)$ and the Grassmannian of pencils of hyperplane sections $G(2,H^0(\OO_X(1)))$, for $X$ in the following list: a quadric $Q^n\subset \p{n+1}$, $G(2,n)$, $G(3,6)$, $OG(4,8)$, $OG(5,10)$, $IG(3,6)$, the Cayley plane $E_6/P_1$, or the Freudenthal variety $E_7/P_7$. The proof relies entirely on representation-theoretic arguments for the restriction of foliations along a minimal equivariant embedding $X\hookrightarrow \PP(V_\lambda)$. Below, we combine the representation theory with stability results for foliations to complete the picture, at least set-theoretically.

\begin{theorem} \label{thm:deg-zero}
Let $X$ be a cominuscule Grassmannian and let $\sF$ be a codimension-one foliation of degree $0$ on $X$. Then $\sF$ is given by a pencil of hyperplane sections. In particular, set-theoretically,
\[
\Fol (X,1,0) \cong  G(2,H^0(\OO_X(1))).
\]
\end{theorem}

\begin{proof}
Note that $c_1(TX) \geq 3$ for any cominuscule Grassmannian. By Corollary \ref{cor:nonvanishing}, we have $H^0(\Omega_X^p(2)) = 0$ for $p \geq 2$. The result then follows from Lemma \ref{lem:integ-stab}; note that $X$ is Fano, thus simply connected.
\end{proof}

\begin{remark}
The theorem above proves that the space $\Fol (X,2)$ of degree-0 foliations on $X$ is set-theoretically a Grassmannian $G(2,H^0(\OO_X(1)))$. In \cite{BFM}, this identification is proved scheme-theoretically. In \cite{Kuster}, the second author proved that for $X$ an adjoint variety not of type A or C, $\Fol (X,2) \cong  G(2,H^0(\OO_X(1)))$ set-theoretically. This holds in particular for $X= OG(2,n)$. However, in \cite[Proposition 6.3]{BFM} it was proved that, scheme theoretically, we have a strict inclusion $\Fol (X,2) \subsetneq  G(2,H^0(\OO_X(1)))$ for $X= OG(2,n)$. One interesting question is to describe this non-reduced scheme structure. 
\end{remark}

\medskip 
\paragraph{Degree-one foliations}
From the techniques exposed so far, we derive the following structure theorem for the space $\Fol (X,3)$ of degree-1 foliations on $X$, a cominuscule Grassmannian that is neither a three-dimensional quadric nor a Lagrangian Grassmannian. Note that $Q^3 \cong IG(2,4)$.

\begin{theorem}\label{thm:deg-one}
Let $X$ be a cominuscule Grassmannian, except $IG(n,2n)$, $n\geq 2$. Let $\sF$ be a codimension-one foliation of degree one on $X$. Then either
\begin{enumerate}
    \item $\sF$ is logarithmic of type $(1,2)$, i.e., there exist $f^2,g\in H^0(\ox(2))$ such that $\sF$ is given by the fibers of $(f^2:g) \colon X \dashrightarrow \p1$; or
    \item $\sF = \phi^*\sG$ is the pullback of a foliation $\sG$ on a surface $S$. Moreover, the fibers of $\phi$ are rationally connected and define a degree-zero foliation with semistable tangent sheaf. 
\end{enumerate}    
\end{theorem}

\begin{proof}
First, $\sF$ is given by $\omega \in H^0(\Omega_X^1(3))$. If $X\neq Q^3 = IG(2,4)$, then $c_1(\tF) = c_1(TX) -3\geq 1$. By \cite[Prop. 7.5]{araujo-druel:fano-foliations} and \cite[Prop. 3.5]{LPT13}, either $\tF$ is semistable and $\sF$ is given by a rational map $(f^2:g) \colon X \dashrightarrow \p1$, with $f^2,g\in H^0(\ox(2))$, or there exists $\sH \subset \sF$ algebraically integrable with rationally-connected leaves such that $c_1(\tH) \geq c_1(\tF)$. In particular, $\sH$ is given by some $\eta \in H^0(\Omega_X^p(k))$ with $k\leq 3$. By Corollary \ref{cor:nonvanishing}, if $X \neq IG(n,2n)$ then $H^0(\Omega_X^p(3)) = 0$ for $p\geq 3$. Therefore, $\tH$ is semistable. The family of leaves of $\sH$ defines a rational map $\phi\colon X \dashrightarrow S$ to a surface $S$ with connected general fiber. It follows that $\sF = \phi^*\sG$ for some foliation $\sG$ on $S$.  
\end{proof}

We expect that the rational map in the second item is the composition of a minimal embedding and a linear projection. This holds if the foliation extends to the ambient projective space of a minimal embedding. However, a proof or a counterexample eludes us at the moment. See \cite{folCompInt} for similar results on complete intersections.

\section{Minimal-degree foliations on homogeneous spaces} \label{S: Examples}
In this section, we present foliations of minimal degree on the classical Grassmannians, the Lagrangian Grassmannians,  the orthogonal Grassmannians, and the Cayley plane $\mathbb{OP}^2$. The main idea is the following. For $X =G/P$, $G$ acts naturally on  $\PP H^0(\Omega_{X}^p(l(p)))$, and $\Fol(X, p, l(p)-p-1)$ is a $G$-invariant subscheme, which is closed since $l(p)$ is the minimum twist to afford global sections. Then it must contain some minimal $G$-orbit, if not empty. We will construct examples of foliations associated with points of a minimal orbit. The natural question that remains open is whether these are the only foliations of this degree.

\subsection{Foliations associated to retangular partitions on ordinary Grassmannians}

Below, we construct examples of foliations in $G(k,n)$ from elements in a minimal orbit.   We will always suppose that $k\leq h:=n-k$ and $V\cong \CC^n$. Recall that $l(p) = \min\{\, l \mid h^0(\Omega_{G(k,n)}^p(l))\neq 0\,\}$ and $d(p) = l(p) -p-1$. 

\begin{theorem}
\label{thm_min_rect}
Let $p=de$ for two integers $d,e$ such that $1\leq e \leq k$, $1\leq d\leq h:=n-k$. Let $V$ be an $n$-dimensional complex vector space. Then 
\[
\Flag(h-d,h+e,V)\subset \Fol(G(k,V),de, d+e - de -1).
\]
More explicitly, given $[\, W_{h-d}\subset W_{h+e} \subset V\,] \in \Flag(h-d,h+e,V)$ one defines a foliation $\sF$ by the fibers of a rational map $\varphi \colon G(k,V) \dashrightarrow G(e, W_{h+e}/W_{h-d})$. The tangent sheaf $\tF$ fits in the following exact sequence
\[
0 \lra W_{h-d} \otimes \U^\vee \lra \tF \lra (V/W_{h+e})^\vee \otimes (V/(\U + W_{h-d}))^{\vee\vee} \lra 0.
\]
In particular, $\sF$ is of minimal degree when $l(p) = d+e$. 
\end{theorem}


\begin{proof}
Fix a flag $[\, W_{h-d} \subset W_{h+e} \subset V \,] \in \Flag(h-d,\, h+e;\, V)$ and consider the rational map
\[
\varphi \colon G(k,V) \dlra G\!\left(e,\, W_{h+e}/W_{h-d}\right),
\]
defined as follows.  Let $\pi \colon W_{h+e} \to W_{h+e}/W_{h-d}$ denote the natural projection. For a general point $U_k \in G(k,V)$, $\dim(U_k\cap W_{h+e}) = e$ and $U_k\cap W_{h-d} = 0$. Then  
\[
\varphi(U_k) = \pi\bigl(U_k \cap W_{h+e}\bigr)
=
\frac{(U_k \cap W_{h+e}) + W_{h-d}}{W_{h-d}}
\]
is well defined. The base locus of $\varphi$ is $\Lambda \cup \Sigma$, where $\Lambda = \{\, U_k \mid \dim(U_k\cap W_{h-d})\geq 1 \,\}$ and $\Sigma = \{\, U_k \mid \dim(U_k\cap W_{h+e})\geq e+1 \,\}$. Note that 
$\Lambda$ and $\Sigma$ represent the Schubert cycles $\sigma_{(d+1)}$ and $\sigma_{(1^{e+1})}$, respectively. Hence,  $\codim \Lambda = d+1$ and $\codim \Sigma = e+1$. For $U_k \in G(k,V) \setminus \{\Lambda \cup \Sigma\}$, the derivative 
\[
D\varphi_{u_k} \colon \Hom(U_k, V/U_k) \lra  \Hom(\varphi(U_k) , W_{h+e}/W_{k-d}/\varphi(U_k)), 
\]
is defined as follows. Note that $(W_{h+e}/W_{k-d})/\varphi(U_k) \cong V/(U_k+W_{h-d})$ and $U_k \cap W_{h-d} =0$. Then let $\rho \colon V/U_k \to V/(U_k+W_{h-d})$ be the canonical projection and $\iota \colon U_k\cap W_{h+e} \to U_k$ be the inclusion. Thus, for $f\in \Hom(U_k, V/U_k)$ we have $D\varphi_{U_k}(f) = \rho\circ f \circ \iota$. Therefore,
\[
\ker D\varphi_{U_k} = \{\, f\colon U_k \to V/U_k \mid f(U_k\cap W_{h+e})\subset U_k+W_{h-d} \,\}. 
\]
Any linear map $\bar{f} \colon V/W_{h+e} \cong {U_k}/({U_k\cap W_{h+e}})  \to V/(U_k+W_{h-d})$ admits a lift $f\in \ker D\varphi_{U_k}$. If $f'$ is another lift of $\bar{f}$, then $f-f' \in \Hom(U_k, W_{h-d})$. Thus, varying $U_k$, we get the short exact sequence over $B = G(k,V) \setminus \{\Lambda \cup \Sigma\}$: 

\[
0 \lra \inhom(\U, W_{h-d} )|_B \lra \ker D\varphi|_B \lra \inhom( V/W_{h+e} ,V/(\U+W_{h-d}))|_B \lra 0
\]
Let $j \colon B \to G(k,V)$ be the inclusion. Then $\tF = j_*\ker D\varphi|_B$ and we get the short exact sequence
\[
0 \lra \inhom(\U, W_{h-d} ) \lra \tF \lra \inhom( V/W_{h+e} ,(V/(\U+W_{h-d}))^{\vee\vee}) \lra 0
\]
Note that $U_k \oplus W_{h-d} \into V \onto V/(\U+W_{h-d})$. Hence $\rk \tF = kh-de$ and  $c_1(\tF) = (h-d) + (k-e) = n-d-e$, that is, $\sF \in \Fol(G(k,V),de, d+e - de -1)$.
\end{proof}

\begin{remark}
In \cite[Example 4.3]{araujo-druel:fano-foliations}, Araujo and Druel describe Fano foliations on $G(k,V)$ as follows. Let $V$ be a vector space of dimension $n$, and let $W \subset V$ be a subspace of dimension $m \leq n-k-1$. The natural projection $V \to V/W$ induces a rational map
\[
G(k,V) \dlra G(k,V/W).
\]
This map defines a foliation $\mathscr{F}$ on $G(k,V)$ of codimension $p = k(n-k-m)$ such that $c_1(\nF) = n-m$. This corresponds to making $e=k$ and $d= n-k-m$ in Theorem \ref{thm_min_rect}. This foliation is of minimum degree when $l(p) = n-m$. 
\end{remark}

\begin{remark}
Notice that in Araujo and Druel examples $\mathscr{F}=\cU^\vee \otimes W_{h-d}$ is automatically reflexive because it is locally free. Another case in which $\mathscr{F}$ is locally free is when $d=h$, in which case $\mathscr{F}=(V/W_{h+e})^\vee \otimes \cQ$. 
\end{remark}

\begin{remark}
The degree computation in the above theorem follows from Lemma \ref{lem:normal-fibration} and $\mathrm{SL}_n$-equivariance. Indeed, using the lemma, one shows that the degree of $\mathscr{F}$ equals $l=d+e-\delta$, where $\delta$ is the ramification divisor. Then, using the fact that the stabilizer of the flag $[W_{h-d}\subset W_{h+e}\subset V]$ permutes transitively all the fibers of $\varphi$, one deduces that there is no ramification in codimension one (i.e. $\delta=0$).

\end{remark}

\begin{Question}
Let $\mathscr{F}$ be a foliation on $X = G(k,V)$ arising from a rectangular partition, as in Theorem~\ref{thm_min_rect}. Suppose that $\mathscr{F}$ is induced by a global section of $H^0(X,\Omega_X^p(p+1))$. 
Can $\mathscr{F}$ be tangent to a codimension-$p$ foliation of degree zero on $X$? If so, is it possible that such a degree-zero foliation is induced by a linear projection $X \dashrightarrow \mathbb{P}^p$?
\end{Question}

\subsection{Minimal-degree foliations on symplectic/orthogonal Grassmannians}
The following result is analogous to Theorem \ref{thm_min_rect} for the orthogonal and symplectic cases. We define $l(p)$ and $d(p)$ as the minimal twist for which there are non-zero holomorphic $p$-forms on symplectic (respectively orthogonal) Grassmannians (as we have done for ordinary Grassmannians).

\begin{theorem}\label{thm:min-examples-orthosympl}
We have the following.
\begin{enumerate}
    \item If $V\cong \CC^{2n}$ is endowed with a symplectic form, and $2p = a(a+1)$, then $l(p) = a+1$, $d(p) =  -\frac{ a(a-1)}{2} $ and
    \[
   IG(n-a;\, V)\subset \Fol(IG(n,V), p, d(p)) \neq \emptyset.
    \]
    \item If $V\cong \CC^{2n}$ is endowed with a quadratic form, and $2p = a(a+1)$, then $l(p) = 2a$, $d(p) = -\frac{(a-1)(a-2)}{2}$ and
    \[
   OG(n-a-1;\, V)\subset \Fol(OG(n,V), p, d(p)) \neq \emptyset.
    \]
\end{enumerate}
\end{theorem}

\begin{proof}
{\bf (Lagrangian Grassmannians)}:\label{ex:sympl} Suppose $2p = a(a+1)$. Then 
\[
H^0\!\left(IG(n,V), \Omega^{p}_{IG(n,V)}(a+1)\right) =V^\vee_{(a+1)\lambda_{n-a}} \cong  V_{(a+1)\lambda_{n-a}},
\]
where $V\cong \CC^{2n}$. Observe that there is a natural inclusion of the symplectic Grassmannian
\(
\mathrm{IG}(n-a;\, V)
\)
into the projective space $\mathbb{P}\!\left(V_{(a+1)\lambda_{n-a}}\right)$.
We claim that each point of this Grassmannian corresponds, under this inclusion, to a foliation on $IG(n,V)$.

Fix a point $[\, W_{n-a}\subset V \,] \in IG(n-a;\, V)$ and consider the rational map
\[
\varphi \colon IG(n,V) \dlra IG\!\left(a,\, W_{n-a}^{\perp}/W_{n-a}\right), \quad [U_n]\mapsto \left[\frac{(U_n\cap W_{n-a}^\perp)+ W_{n-a}}{W_{n-a}}\right].
\]
Here we have endowed the $2a$-dimensional space $W_{n-a}^{\perp}/W_{n-a}$ with the symplectic form which is the restriction of the symplectic form on $V$. Let $\sF$ be the foliation on $IG(n,V)$ induced by the rational map $\varphi$. The general fiber of $\varphi$ over $[A_{a}]\in IG(a,\, W_{n-a}^\perp/W_{n-a})$ can be described as in the proof of Theorem \ref{thm_min_rect}; it is birational to the Lagrangian Grassmannian bundle $ IG(n-a,\cU_{a}^\perp/\cU_a) $ over $G(a,B_{n})$, where $B_{n}:=A_{a}+W_{n-a}$. The base locus of $\varphi$ is where either $\dim(U_n \cap W_{n-a}^\perp) \geq a+1$ or $\dim(U_n\cap W_{n-a}) > 0$. As in the proof of Theorem \ref{thm_min_rect}, one considers an isotropic flag and computes Schubert cycles to determine that the base locus of $\varphi$ has codimension $a+1$. 
We observe that, by construction,
\(
\varphi^{*}\mathcal{O}_{IG(a,2a)}(1) \cong \mathcal{O}_{IG(n,V)}(1).
\)
Again, we have that the map $\varphi$ has no ramification divisor. We therefore conclude that $\sF$ has codimension $p$ and $c_1(\nF) = a+1$. The degree is $a+1 -p-1 = -\frac{a(a-1)}{2}$.
\medskip

\noindent {\bf (Spinor varieties)}:\label{ex:ortho} Suppose $2p = a(a+1)$. Then 
\[
H^0\!\left(OG(n,V), \Omega^{p}_{OG(n,V)}(2a)\right) =V^\vee_{a\lambda_{n-a-1}} \cong  V_{a\lambda_{n-a-1}},
\]
where $V\cong \CC^{2n}$. Observe that there is a natural inclusion of the orthogonal Grassmannian
\(
\mathrm{OG}(n-a-1;\, V)
\)
into the projective space $\mathbb{P}\!\left(V_{a\lambda_{n-a-1}}\right)$.
We claim that each point of this Grassmannian corresponds, under this inclusion, to a foliation on $OG(n,V)$.

Fix a point $[\, W_{n-a-1}\subset V \,] \in OG(n-a-1;\, V)$ and consider the rational map
\[
\varphi \colon OG(n,V) \dlra OG\!\left(a+1,\, W_{n-a-1}^{\perp}/W_{n-a-1}\right), \quad [U_n]\mapsto \left[\frac{(U_n\cap W_{n-a-1}^\perp)+ W_{n-a-1}}{W_{n-a-1}}\right].
\]
Here we have endowed the $2a+2$-dimensional space $W_{n-a-1}^{\perp}/W_{n-a-1}$ with the quadratic form which is the restriction of the quadratic form on $V$. Let $\sF$ be the foliation on $OG(n,V)$ induced by the rational map $\varphi$. The general fiber of $\varphi$ over $[A_{a+1}]\in OG(a+1,\, W_{n-a-1}^\perp/W_{n-a-1})$ can be described as in the proof of Theorem \ref{thm_min_rect}; it is birational to the Spinor Grassmannian bundle $ OG(n-a-1,\cU_{a+1}^\perp/\cU_{a+1}) $ over $G(a+1,B_{n})$, where $B_{n}:=A_{a+1}+W_{n-a-1}$. The base locus of $\varphi$ is where either $\dim(U_n \cap W_{n-a-1}^\perp) \geq a+2$ or $\dim(U_n\cap W_{n-a-1}) > 0$. As in the proof of Theorem \ref{thm_min_rect}, one considers an isotropic flag and computes Schubert cycles to determine that the base locus of $\varphi$ has codimension $a+2$. 
By construction,
\(
\varphi^{*}\mathcal{O}_{IG(a+1,2a+2)}(1) \cong \mathcal{O}_{OG(n,V)}(1).
\)
Again, we have that the map $\varphi$ has no ramification divisor. We therefore conclude that $\sF$ has codimension $p$ and $c_1(\nF) = 2a$. The degree is $2a-\frac{a(a+1)}{2} -1 = -\frac{(a-1)(a-2)}{2}$.
\end{proof}

\begin{remark} \label{Remark deg 1 simpletic}
Theorem \ref{thm:deg-one} is false for $X = IG(n,2n)$. For $n\geq 3$, one has that $H^0(\Omega_X^3(3)) \neq 0$ and contains integrable forms as in the proof of Theorem \ref{thm:min-examples-orthosympl}. These foliations are given by rational maps $\varphi \colon X \dashrightarrow IG(2,4) \cong Q^3$, the tridimensional quadric. Therefore, one obtains codimension-one foliations of degree one on $IG(n,2n)$ of the form $\varphi^*\sG$, where $\sG$ is a foliation on $Q^3$ induced by an action of a noncommutative Lie algebra, see \cite[\S5]{LPT13}. These foliations do not fit in either item of Theorem \ref{thm:deg-one}, as shown in Theorem \ref{T: IG(n,2n) components} below. 

Corollary \ref{cor:nonvanishing} and Lemma \ref{lem:integ-stab} imply that integrable forms in $H^0(\Omega_X^3(3))$ define algebraically integrable foliations with rationally connected leaves. One then expects that these are precisely the foliations in the proof of Theorem \ref{thm:min-examples-orthosympl} for $p=3$. However, a proof or counterexample eludes us for the moment.
\end{remark}
 
\begin{theorem} \label{T: IG(n,2n) components}
For every $n\geq 2$,  $\Fol(IG(n,2n),1,1)$ has at least three irreducible components.
\end{theorem}

\begin{proof}
Logarithmic foliations form two irreducible components 
\[
\Log(1,2) \cup \Log(1,1,1) \subset \Fol(IG(n,2n),1,1),
\]
see \cite{Calvo1994,GMLN-rat}. We only need to exhibit a foliation that is not in a logarithmic component. 

Let $\varphi \colon X \dashrightarrow  Q^3$ be as in Theorem \ref{thm:min-examples-orthosympl} and let $\sG$ be the foliation described in \cite[Example 5.1]{LPT13}. Then define $\sF = \varphi^*\sG$. The foliation $\sG$ leaves invariant an irreducible surface $S \in |\OO_{{Q}^3}(3)|$. It follows that $\varphi^*(S) \in |\OO_{X}(3)|$ is irreducible. Indeed, suppose that $\varphi^*(S) = A \cup B$ with $A\neq B$. Note that $A$ nor $B$ cannot be the pullback of a divisor on $Q^3$,  so we may assume that $A$ is not a pullback.   Then take $C = \varphi^*(H)$ for some general $H\in |\OO_{{Q}^3}(1)|$ and define $D = a\, A+ b\,C$ so that $c_1(D) = 0$. 

Since $D$ is $\varphi$-invariant and $H^0(\Omega_X^4(3)) = 0$, \cite[Lemma 3.1]{LPT13} implies that $\varphi$ is tangent to a codimension-one logarithmic foliation $\sH$ defined by a rational $1$-form $\eta$ with poles along $A \cup C$. In other words, the general fibers of $\varphi$ are tangent to $\sH$. Applying \cite[Lemma 2.4]{loray-pereira-touzet:trivial}, it follows that $\eta$ is the pullback via $\varphi$ of a rational $1$-form $\theta$. But this contradicts the choice of polar divisor of $\eta$ (not a pullback from $Q^3$). Thus $\varphi^*(S)$ must be irreducible. 
Finally, the same argument as in \cite[Example 5.1]{LPT13} shows that $\sF$ is not in the closure of the logarithmic components. 
\end{proof}

\subsection{Minimal-degree foliations on the exceptional varieties}
To conclude this section, we discuss a natural example of minimal-degree foliation on the Cayley plane $\OOP^2 = E_6/P_1 \subset \p{26}$. 

From Table \ref{tab:cohE6}, one notices that $H^0(\Omega_{\OOP^2}^8(8)) = V_{4\lambda_1}$ hence the minimal orbit is isomorphic to $\OOP^2$; and there is a way to associate a point of $\OOP^2$ to a codimension-8 foliation. The Cayley plane is covered by octonionic lines $\OOP^1$, which are isomorphic to the quadric $Q^8$. The family of octonionic lines is parametrized by the dual plane $\overline{\OOP^2} = E_6/P_6 \cong \OOP^2$. For $y\in \overline{\OOP^2}$ we denote by $Q_y\subset \OOP^2$ the corresponding octonionic line, and similarly for $x\in \OOP^2$ and $Q_x \in \overline{\OOP^2}$. The following properties hold, see \cite{im_cayley} and references therein.
\begin{enumerate}
    \item For any $x,x' \in \OOP^2$ such that the usual line through $x$ and $x'$ is not contained in $\OOP^2$ there exists a unique $y\in \overline{\OOP^2}$ such that $Q_y \ni x,x'$;
    \item If $y,y' \in \overline{\OOP}^2$ such that the the usual line through $y$ and $y'$ is not contained in $\overline{\OOP}^2$ then $Q_y \cap Q_{y'}$ is a unique reduced point.
\end{enumerate}
Given a point $x\in \OOP^2$ one obtains a rational map 
\[
\varphi \colon \mathbb{OP}^2 \dlra \mathbb{OP}^1\quad , \quad x'\mapsto y\in Q_x \cong \mathbb{OP}^1. 
\]

\begin{proposition}\label{prop:fol-Cayley}
The foliation $\sF$ given by the fibers of the map $\varphi$ above satisfies $c_1(\nF) = 8$. Hence $\sF$ is given by an integrable form in $H^0(\Omega_{\OOP^2}^8(8))$.
\end{proposition}

For simplicity, we denote $\varphi^{-1}(\,\cdot\,)$ the closure $\overline{\varphi|_U^{-1}(\,\cdot\,)}$ in $\OOP^2$, where $U\subset \OOP^2$ is the domain of definition of $\varphi$.

\begin{proof}
We will show that $\varphi^*\OO_{\OOP^1}(1) = \OO_{\OOP^2}(1)$. The result then follows from Lemma \ref{lem:normal-fibration}. Recall the explicit construction of $Q_x$. Consider the affine tangent space $\widehat{T_{x}\OOP^2} \subset \CC^{27}$; then (the projectivized of) its orthogonal in the dual projective space is a nine-dimensional projective space $\PP(\widehat{T_{x}\OOP^2}^\perp)$ that cuts out $\overline{\OOP^2}$ at $Q_x$. Now, consider a point $[y]\in Q_x$. Such a point defines a hyperplane $\PP(y^{\perp_q})$ in $\PP(\widehat{T_{x}\OOP^2}^\perp)$ (here $q$ is the quadratic form defining $Q_x$), hence a hyperplane section $Y$ in $Q_x$. We want to show that $\varphi^{-1}(Y)$ is $\PP(y^\perp)\cap \mathbb{OP}^2$, a hyperplane section in $\mathbb{OP}^2$. Since $\varphi^{-1}(Y)$ must contain a divisor, it is sufficient to show that $\varphi^{-1}(Y)\subset \PP(y^\perp)$. Let $[z]\in Y$, i.e. $z\perp_q y$, $[z]\in Q_x$. Since $z^{\perp_q}=\widehat{T_zQ_x}$, we deduce that $y\in \widehat{T_zQ_x}$. Moreover $Q_x\subset \overline{\mathbb{OP}^2}$ so that $y\in \widehat{T_zQ_x}\subset \widehat{ T_z\overline{\mathbb{OP}^2}}$. Taking orthogonals in the dual space the inclusions get reversed, so that 
\[
\PP\!\left(\widehat{ T_z\overline{\mathbb{OP}^2}}^\perp\right) \subset \PP\!\left(\widehat{T_zQ_x}^\perp\right)\subset \PP(y^\perp).
\]
This implies that 
\[
Q_z=\PP\!\left(\widehat{ T_z\overline{\mathbb{OP}^2}}^\perp\right) \cap \mathbb{OP}^2 \subset \PP(y^\perp)\cap \mathbb{OP}^2.
\]
Since the fiber $\varphi^{-1}([z])=Q_z$ and the above inclusions hold for any $[z]\in Y$, we deduce that $\varphi^{-1}(Y)\subset \PP(y^\perp)\cap \mathbb{OP}^2$, and hence they coincide by dimensional reasons and the irreducibility of $\PP(y^\perp)\cap \mathbb{OP}^2$. Therefore $\varphi^*\cO_{Q_x}(1)=\cO_{\mathbb{OP}^2}(1)$.   
\end{proof}

\bibliographystyle{alphaurl}
\bibliography{biblio}
\end{document}